\documentclass[12pt, notitlepage]{amsart}
\usepackage{latexsym, amsfonts, amsmath, amssymb, amsthm, cite}

\newtheorem{theorem}{Theorem}[section]
\newtheorem{lemma}[theorem]{Lemma}
\newtheorem{proposition}[theorem]{Proposition}

\newtheorem{corollary}[theorem]{Corollary}

\newtheorem{definition}[theorem]{Definition}

\usepackage{graphicx}
\usepackage{mathrsfs}

\begin{document}

\title{Integral Factorial ratios: Irreducible examples with height larger than $1$} 
 \author{K. Soundararajan} 
\address{Department of Mathematics \\ Stanford University \\
450 Serra Mall, Bldg. 380\\ Stanford, CA 94305-2125}
\email{ksound@math.stanford.edu}
\thanks{This work is partially supported by a grant from the National Science Foundation, and a Simons Investigator award from 
the Simons Foundation.} 
\dedicatory{In celebration of the 100th anniversary of Ramanujan's election to the Royal Society}
 \maketitle

\section{Introduction}  

\noindent This article is concerned with the problem of classifying tuples of natural numbers $a_1$, $\ldots$, $a_K$ and 
$b_1$, $\ldots$, $b_L$ with $\sum_i a_i =\sum_j b_j$ such that for all natural numbers $n$ one has 
$$ 
\frac{(a_1 n)! (a_2 n)! \cdots (a_K n)!}{(b_1 n)! (b_2 n)! \cdots (b_L n)!} \in {\Bbb N}.    
$$ 
Clearly we may assume that no $a_i$ equals $b_j$.   Further, it turns out that 
there are no solutions unless $L >K$, and that one can restrict attention to {\sl primitive} tuples such that the gcd of $(a_1, \ldots, a_K, b_1, \ldots, b_L)=1$.   
The condition $\sum_i a_i = \sum_j b_j$ guarantees that the factorial ratios grow only exponentially with $n$, so that the power series formed with these 
coefficients is a hypergeometric series.    We have in mind the situation when $D=L-K$ is a fixed positive integer, which is called the {\sl height} of the factorial ratio.

The general problem of describing such factorial ratios is largely open, with a complete solution being available only in the case of height $1$.  In this case, Rodriguez-Villegas \cite{RV} made 
the fundamental observation that the integrality of the factorial ratio is equivalent to the algebraicity of the associated hypergeometric function.  The work of Beukers and Heckman \cite{BH} gave a complete classification of such algebraic hypergeometric functions (which correspond to the instances where the associated monodromy group is finite).   This connection was made precise by Bober \cite{Bober, Bober2}, who showed that for $D=1$ there are three infinite families and fifty two sporadic examples.  One of these sporadic examples goes back to Chebyshev in 
connection with his works on prime numbers: for all $n\in {\Bbb N}$ 
$$ 
\frac{(30n)! n!}{(15n)! (10n)! (6n)!}  \in {\Bbb N}. 
$$ 
Bober's work confirmed a conjecture of Vasyunin  \cite{V} who had identified the three 
infinite families and fifty two sporadic examples in connection with a problem motivated by the Nyman--Beurling equivalent formulation of the Riemann Hypothesis.  In the recent paper \cite{S}, I gave a new elementary proof of the classification in the case $D=1$, which is independent of 
the results of Beukers and Heckman, and made partial progress on understanding larger values of $D$.

In this article we shall give a number of new examples of factorial ratios with $D\ge 2$.   Trivially, one can take two factorial ratios with $D=1$ and multiply these 
together to obtain an example with $D=2$.  The examples we give will be shown to be {\sl irreducible}; that is, not to arise in this fashion.  In particular, for $D=2$ we shall 
give more than fifty examples of irreducible two parameter families of integral factorial ratios.  Here is one such two parameter family: if $a$ and $b$ are coprime natural numbers 
with $a \ge 5b$ then for all $n\in {\Bbb N}$ we have 
$$ 
\frac{(6an)! (bn)!}{(2an)! (3an)! (6bn)! ((a-5b)n)! } \in {\Bbb N}. 
$$ 
Taking $b=1$, $a=5$ in the above example leads to the Chebyshev example with $D=1$.

Before we can describe our work, we must recapitulate the notation and some of the results from our earlier paper \cite{S}.   Let $\frak a = [a_1,\ldots, a_n]$ 
denote a list of $n$ non-zero integers.   We shall always assume that our lists are non-degenerate in the sense that $\frak a$ does not contain a pair 
of elements $a$, $-a$.   Given a non-degenerate list $\frak a$, we denote by $\ell(\frak a)$ the length of the list, by $s(\frak a)$ the sum of the elements 
$a_1+\ldots +a_n$, and by $h(\frak a)$ its height which is defined as the number of negative elements in $\frak a$ minus the number of positive elements.   
We call a list primitive if the gcd of all its elements equals $1$.   The order of elements in lists is irrelevant, and we will treat all 
permutations of a list as being the same.  Also, given a non-zero integer $k$, we denote by $k\frak{a}$ the list obtained by multiplying every element of $\frak a$ 
by $k$.

Let $\{ x\}  = x-\lfloor x\rfloor$ denote the fractional part of $x$, and let $\psi(x) = 1/2- \{x\}$ denote the ``saw-tooth function".   
To a list $\frak a$ we associate a $1$-periodic function ${\frak a}(x)$, defined as follows.  If $a_j x \not \in {\Bbb Z}$ for all $j$, put 
\begin{equation} 
\label{1.1} \frak a(x) = \sum_{j=1}^{n} \psi(a_j x), 
\end{equation} 
and extend $\frak a(x)$ to the remaining points by right continuity: $\frak a(x) = \frak a(x^+)$.  
We also define the ``norm" of $\frak a$ (which played a central role in the investigations of \cite{S}) by 
\begin{equation} 
\label{1.2} 
N([a_1,\ldots, a_n] ) = N(\frak a) = \int_0^1 {\frak a}(x)^2 dx = \frac 1{12} \sum_{i, j=1}^n \frac{(a_i, a_j)^2}{a_i a_j}. 
\end{equation}  
The last identity above follows from an easy calculation using Parseval's formula and the Fourier expansion of the saw-tooth function; see 
(2.1) of \cite{S}.

If $(a_1,\ldots, a_K, b_1,\ldots, b_L)$ is a $K+L$--tuple of natural numbers corresponding to an integral factorial ratio of height $D=L-K$, then we 
associate to this tuple the list $\frak a = [a_1, \ldots, a_K, -b_1, -b_2, \ldots, -b_L]$ which is a non-degenerate list with $\ell(\frak a) = 
K+L =2K +D$, $h(\frak a)=D$, and $s(\frak a) =0$.   The integrality of the factorial ratio is equivalent to the condition that $\sum_{i=1}^{K} \lfloor a_i x\rfloor 
- \sum_{j=1}^{L } \lfloor b_j x\rfloor \ge 0$ for all real numbers $x$.  This observation goes back to Landau, and is based on comparing the 
power of a prime $p$ dividing the numerator and denominator of the factorial ratio.  More precisely, the integrality of the factorial ratio is 
equivalent to $\sum_{i=1}^{K} \lfloor a_i x\rfloor 
- \sum_{j=1}^{L } \lfloor b_j x\rfloor $  taking values in the set $\{ 0, 1, \ldots, D\}$ for all real $x$, which is the same as requiring $\frak a(x)$ 
to take values in the set $\{ -D/2 +k: \ \ 0\le k\le D \}$.  Here it may be useful to note that $\sum_{i=1}^{K} \lfloor a_i x \rfloor - \sum_{j=1}^{L} \lfloor b_j x\rfloor$ is 
right continuous, which motivated our prescription of right continuity for $\frak a(x)$.

In the case $D=1$, the function $\frak a(x)$ is constrained to take just the two values $-1/2$ and $1/2$.  This permits an elegant characterization of integral factorial 
ratios of height $1$:  these correspond to lists $\frak a$ with odd length $\ell(\frak a)$, height $h(\frak a)=1$,  sum $s(\frak a)=0$, and with norm $N(\frak a) =1/4$.   Therefore the 
norm is a particularly valuable tool in understanding factorial ratios of height $1$, and forms the basis for the classification of such ratios in \cite{S}.    When $D\ge 2$, the 
norm alone does not characterize integral factorial ratios, but nevertheless it forms a useful starting point for the investigation of that problem.  If $\frak a$ corresponds to an 
integral factorial ratio with height $D$, then its norm $N(\frak a)$ must be $\le D^2/4$.  In \cite{S}, we showed (using this observation) that if $D=2$ then $K+L \le 80$, and 
that the points $(a_1,\ldots, a_K, b_1, \ldots, b_L) \in {\Bbb R}^{K+L}$ lie on finitely many vector subspaces of ${\Bbb R}^{K+L}$ of dimension at most $11$.

 Given two lists $\frak a_1$ and $\frak a_2$, we denote by $\frak a_1 + \frak a_2$ the list obtained by concatenating the two lists and removing any degeneracies.    
 If $\frak a_1$ and $\frak a_2$ correspond to integral factorial ratios with height $D_1$ and $D_2$ then $\frak a_1  + \frak a_2$ corresponds to an integral factorial ratio 
 of height $D_1+ D_2$.    This gives a trivial way of constructing factorial ratios of height larger than $1$, and the following definition is an attempt to distinguish such 
 examples from genuinely new examples of height larger than $1$.

\begin{definition}  A list $\frak a$ corresponding to a factorial ratio with height $D$ is called reducible if $\frak a = \frak b+ \frak c$ and 
$\frak b$ and $\frak c$ correspond to factorial ratios with smaller heights.   If $\frak a$ cannot be reduced in that way, then $\frak a$ is called 
irreducible. 
\end{definition} 

For example, the lists $[(a+b+c), -a, -b, -c]$ with $a$, $b$, $c$ being positive integers, correspond to multinomial coefficients, and thus give
examples of integral factorial ratios with height $2$.  These lists are all reducible since they may be decomposed as $[(a+b+c), -a, -(b+c)] + [(b+c), -b, -c]$; 
that is, the multinomial coefficient can be expressed as a product of binomial coefficients.  

We are now ready to present our results.  The first result provides a classification of all lists with height $2$ (that is, with two more negative entries than positive) 
and norm at most $1/3 +\delta$ for some small $\delta >0$. 

\begin{theorem}  \label{thm1} Let $\frak a$ be a primitive list of height $2$ with $N(\frak a) \le 1/3+\delta$ for some small $\delta >0$.  Then, apart from finitely many lists, 
$\frak a$ belongs to one of $28$ families described explicitly in Section 3.   There are two three parameter families, and $26$ two parameter families.   Sixteen 
of the families are reducible in the sense that every list of height $2$ in this family is a reducible list, and the remaining $12$ are irreducible in the sense that they 
contain infinitely many primitive lists of height $2$ that are irreducible.   All lists with height $2$ in the $28$ families give examples of integral factorial ratios with height $2$.  
\end{theorem}  

In theory it would be possible to determine the finitely many lists left unspecified in our theorem, but this might be computationally demanding (or even 
infeasible).   Just as the lists in the infinite families all gave examples of factorial ratios, we hazard the guess that the same property holds for the finitely many 
lists also; in other words, every primitive list of height $2$ and norm at most $1/3+ \delta$ for some small $\delta >0$ gives rise to an integral factorial ratio.    In contrast, a typical 
list of height $2$ from the family $[3a,18a,-a,-9a,-b,-11a+b]$ has norm very nearly $37/108=0.34259\ldots$, but one can find many examples in this family that do 
not correspond to integral factorial ratios (indeed, we believe that there are only finitely many primitive lists in this family that are integral factorial ratios).  

The other main result of this paper gives a way of constructing integral factorial ratios of height larger than $1$, and we shall use this method to exhibit many 
irreducible two parameter families with height $2$.  

\begin{definition}  A list $\frak b = [b_1, \ldots, b_k]$ is called monotone if the associated function $\sum_{i=1}^{k} \lfloor b_i x\rfloor$ (defined thus if $b_i x \not \in {\Bbb Z}$ for 
all $i$, and extended by right continuity to all $x$) is a monotone function of $x$.  
If $s(\frak b)$ is positive, then this associated function is monotone increasing, and if $s(\frak b)$ is negative then it is monotone decreasing. 
\end{definition} 

\begin{theorem} \label{thm2}  Suppose $\frak a$ and $\frak b$ are primitive lists, with $\frak b$ monotone, and such that $s(\frak a)$ and $s(\frak b)$ are both 
non-zero with $(s(\frak a), s(\frak b)) =1$.  
Suppose $s(\frak b) \frak a + (-s(\frak a)) \frak b$ is a list of height $D$ corresponding to an integral factorial ratio.   Then the lists with height $D+1$ that belong to the family 
$$ 
a \frak a + b \frak b + (-a s(\frak a) - bs(\frak b)) [1] 
$$ 
are integral factorial ratios.  
\end{theorem}  

It is easy to check that the lists $[1,-k]$ (for any integer $k\ge 2$), $[1,-2,-k,2k]$ (for odd $k\ge 3$), $[1,-2,k]$ (for even $k\ge 4$), $[-1,2,3]$, $[2,-3,-4]$ are 
all monotone, and using these together with a knowledge of factorial ratios with height $1$, we give in Section 5 many examples of two parameter 
families of height $2$ arising  from Theorem \ref{thm2}.  Starting with these examples, and using Theorem \ref{thm2} repeatedly, one can produce 
three parameter families with height $3$ and so on.   In Section 6 we discuss the structure of reducible lists with height $2$, and use this 
to show that the examples produced in Section 5 with height $2$ along with the $12$ families mentioned in Theorem \ref{thm1} are all irreducible.

Finally, let us mention some other examples of integral factorial ratios with height larger than $1$.   In a {\sl Monthly} problem, Askey \cite{A} gives the two parameter, height $2$ family 
$[3 (m+n), 3n, 2m, 2n, -(2m+3n),-(m+2n),-(m+n),-m, -n ,-n]$, which we checked is irreducible using our work in Section 6.    Askey's example arose in the 
context of the Macdonald--Morris conjectures,  which in this context was resolved in the work of Zeilberger \cite{Z}.    
The Macdonald--Morris conjectures are intimately connected with the theory of the Selberg integral and provide further examples of integral factorial ratios; see \cite{FW} for further information in 
this direction.  For example, the root system $BC_n$ gives rise to a three parameter factorial ratio of height $n$ (see page 501 of \cite{FW}), giving in 
particular a three parameter family of height $2$.   Gessel \cite{G} discusses finding integral factorial ratios 
via combinatorial arguments.  In particular, Gessel gives a $4$ parameter family of height $3$ --  namely, $[(k+2\ell),(k+2m),(k+2n),(k+\ell+m+n), -k, -\ell, -m , -n, -(k+\ell+m),-(k+\ell+n), -(k+m+n)]$ -- along with several examples of $3$ parameter families of height $2$.  Wider \cite{W} gives examples of integral factorial ratios of height larger than $1$, and discusses the problem of showing whether such examples are reducible or not.  He gives the height $2$ family $[3a, -a, 3b, -b, -(a+b), -(a+b)]$, which he shows is irreducible.

 \section{Toward the proof of Theorem \ref{thm1}} 
 
 \noindent In this section we recall some results from \cite{S} on identifying lists with small norm, and set the stage for the proof of Theorem \ref{thm1}.  
 A key notion developed in \cite{S} is that of $k$-separated lists; see Definition 2.1 of \cite{S}.  Briefly, a primitive list $\frak a$ is called $k$-separated if there are 
 two primitive lists $\frak b$ and $\frak c$ with $1\le \ell(\frak b), \ell(\frak c) < \ell(\frak a)$ such that the following properties hold.   There are non-zero coprime integers $B$ and $C$ 
 such that $\frak a = B\frak b + C\frak c$.  Exactly one of $B$ or $C$ is divisible by $k$, and the other is coprime to $k$.  If $k|B$, then for all $kb \in B\frak b$ and $c\in C\frak c$ we 
 have $(kb,c) = (b,c)$, and an analogous criterion holds if $k|C$.    For a more fleshed out discussion of this definition, and examples, we refer to \cite{S}.  
 
 There are two key properties of this definition.  First, one can compute the norm of $\frak a$ in terms of the norms of $\frak b$ and $\frak c$:  in particular, one has from 
 Proposition 2.2 of \cite{S} that $N(\frak a) \ge (1-1/k)(N(\frak b) + N(\frak c))$.   Second, given $n$ and $k$, there are only finitely many primitive lists of length $n$ that 
 are at most $k$--separated (which means that the list is not $\ell$ separated for any $\ell >k$); see Proposition 2.4 of \cite{S}.  
 
 These two properties enable an inductive approach to classifying lists of small norm, and we now extract from \cite{S} some conclusions in this regard.

  \begin{lemma} \label{lem2.1}   Let $\frak a$ be a primitive list.   Then $N(\frak a) \ge 31/180$ except in the following cases: 
  $$
  \frak a = [1], \qquad [a,b], \qquad [a,-2a, b], \qquad [a,-2a, b,-2b],
  $$  
  $$
  \text{Norm  } \tfrac{17}{108}:  \qquad [1,-3,9], \qquad [1,-2, -3, 6,9 ,-18];
  $$ 
  $$
  \text{Norm  } \tfrac{1}{6}: [1, -3, -4], \ \ [3, 4, -12], \ \ [1, - 3, 6, -12], \ \ [1, -3, -4, 6], \ \ [1, -3, -4, 12], \ \ 
  $$
  $$
  [1,-2,4,-12], \ \ [1,-2,-3,4], \ \ [1,-2,-3,12], \ \ [1,-4,-6,12], \ \ [2,-3,-4,12], \ [3,-4,-6,12], \\  
  $$
  $$
  [1,-2,-3,4,6], \ \ [1,-2,-3,6,-12], \ \ [2,3,-4,-5,12], \ [1, -2,4,6,-12], \ [1,-2,-3,4,6,-12].
  $$
 \end{lemma} 
 \begin{proof}  Clearly $N([1])=1/12$, and the three families $[a,b]$, $[a,-2a, b]$, and $[a,-2a, b,-2b]$ 
 give lists with norms close to $1/6$ once $|a|$ or $|b|$ is sufficiently large (and with $(a,b)=1$).  The remaining catalog of 
 lists follows from the work in \cite{S}: see there Section 4.3, Lemmas 4.2, 7,1, 7,3, 7.4, together with the bounds for $G(n)$ discussed 
 in Section 3. 
 \end{proof} 
 
 For future use, let us also record the first few smallest norms that are possible:
 $$
 \text{Norm } \tfrac{1}{12}:  [1], \ [1,-2];  \ \ \text{Norm } \tfrac{1}{9}: [1,-3],  \ [1,-2, -3, 6]; \ \ \text{Norm } \tfrac 18:  [1, -4], \ [1,-2,4].
 $$

 \begin{lemma} \label{lem2.2}  Let $\frak a$ be a primitive list with $s(\frak a)=0$.   If $\ell(\frak a)$ is odd, then $N(\frak a) \ge 1/4$.   If $\ell(\frak a)$ is even, 
 then $N(\frak a) \ge 31/180$ except for the two lists $[1,-2,-3,4]$ and $[1,-3,-4,6]$ which have norm $1/6$.   
 \end{lemma} 
 \begin{proof}  Note that if $\frak a$ is primitive with $s(\frak a) =0$ then $\ell(\frak a) \ge 3$.  If $\ell(\frak a)$ is odd, then $\frak a(x)$ takes values $k+1/2$ for an integer 
 $k$, which implies that $N(\frak a) \ge 1/4$.  If $\ell(\frak a)$ is even, then the lemma follows upon examining the lists in Lemma \ref{lem2.1}.  
 \end{proof}

 \begin{proposition} \label{prop2.1}   Let $\frak a$ be a primitive list with norm $\le 1/3+ \delta$, $s(\frak a) =0$, and $h(\frak a)=2$.   Then, apart from 
 finitely many exceptional lists, $\frak a$ lies in a space of the form 
 $$ 
 x_1 \frak a_1 + x_2 \frak a_2 + x_3 \frak a_3, \qquad\text{ with } x_1 s(\frak a_1) + x_2 s(\frak a_2) + x_3 s( \frak a_3) =0, 
 $$ 
 where $\frak a_1$, $\frak a_2$, $\frak a_3$ are primitive lists with $s(\frak a_i) \neq 0$ and $N(\frak a_1) + N(\frak a_2) + N(\frak a_3) \le 1/3+2\delta$.   
 \end{proposition} 
 \begin{proof}   From \cite{S} we know that if $\ell(\frak a)$ is sufficiently large, then $N(\frak a)$ is also large; for example, if $\ell(\frak a) \ge 82$ then 
 $N(\frak a) >1$.   Thus we can restrict attention to lists $\frak a$ with bounded length, and so after excluding finitely many primitive lists, we may 
 assume that $\frak a$ is $k$-separated for some $k \ge 2/\delta$.  Thus, by the definition of $k$-separated, 
 we can find two primitive lists $\frak b$ and $\frak c$ such that $\frak a = B\frak b + C\frak c$, and $N(\frak b) + N(\frak c) \le N(\frak a) (k/(k-1)) 
 \le N(\frak a)(1+\delta)$.   
 
 If either $s(\frak b)$ or $s(\frak c)$ is zero, then (because $s(\frak a)=0$) the other must also be zero.   
If $s(\frak b) = s(\frak c) =0$, then Lemma \ref{lem2.2} would imply that $\frak b$ and $\frak c$ would have to be either $[1,-2,-3,4]$ or $[1, -3, -4, 6]$, 
but in all these cases it is not possible for $\frak a = B\frak b + C\frak c$ to have height $2$.  

Therefore, we may suppose that $s(\frak b)$ and 
$s(\frak c)$ are both non-zero.   Since $\frak a = B\frak b + C\frak c$ is primitive, we must have $C= \pm s(\frak b)/(s(\frak b), s(\frak c))$ and $B = 
\mp s(\frak c)/(s(\frak b),s(\frak c))$, so that $\frak a$ is determined uniquely by $\frak b$ and $\frak c$.    If both $\frak b$ and $\frak c$ are 
at most $\lceil 2/\delta \rceil$-separated, then there would only be finitely many possibilities for $\frak b$ and $\frak c$, and therefore only 
finitely many choices for $\frak a$.  

Suppose then that $\frak c$ is at least $2/\delta$--separated.   Now $\frak c$ must decompose as $\frak c = D\frak d + E \frak e$, where 
$\frak d$ and $\frak e$ are primitive lists with $N(\frak d) + N(\frak e) \le N(\frak c)(1+\delta)$.   Renaming $\frak b$ as $\frak a_1$, $\frak d$ as 
$\frak a_2$, and $\frak e$ as $\frak a_3$ we conclude that $\frak a$ is of the form $x_1 \frak a_1 + x_2 \frak a_2 + x_3 \frak a_3$ as desired, and 
that $N(\frak a_1) + N(\frak a_2) + N(\frak a_3) \le N(\frak b) + (1+\delta) N(\frak c) \le (1+\delta)^2 N(\frak a) \le \frac 13+2\delta$.

There is one final remaining point.  We know that $s(\frak a_1) (=s(\frak b)) \neq 0$, but it is conceivable that one of $s(\frak a_2)$ or $s(\frak a_3) =0$; 
say, $s(\frak a_3)=0$.  To rule this scenario out, note that by Lemma \ref{lem2.2} we must then have $\frak a_3 = [1,-2,-3,4]$ or $[1,-3,-4,6]$ and 
thus $N(\frak a_3)=1/6$.  This forces $N(\frak a_1) + N(\frak a_2) \le 1/6+2\delta$, which implies that $\frak a_1$ and $\frak a_2$ must be $[1]$ or $[1,-2]$.   
Since $\frak a$ has height $2$, we are further forced to have $\frak a_1 = \frak a_2 =[1]$, but now we must have $x_1 = -x_2$ in order to have $s(\frak a) =0$, 
and the resulting $\frak a$ has height $0$.  Therefore $s(\frak a_i) \neq 0$ for $i=1$, $2$, $3$, and the proof of the proposition is complete. 
 \end{proof}

 \section{Proof of Theorem \ref{thm1}:  Restricting to $28$ families}  
 
 \noindent If $\frak a$ is a primitive list of height $2$ with $s(\frak a)=0$ and  $N(\frak a) \le 1/3+\delta$, then apart from finitely many 
exceptions, we know (by Proposition \ref{prop2.1}) that $\frak a$ is of the form 
\begin{equation} 
\label{3.1}  x_1 \frak a_1 + x_2 \frak a_2 + x_3 \frak a_3, \qquad \text{with } x_1 s(\frak a_1) + x_2 s(\frak a_2) + x_3 s(\frak a_3) = 0,
\end{equation} 
where $N(\frak a_1) + N(\frak a_2) + N(\frak a_3) \le 1/3+ 2\delta$.  In this section, we classify all the possibilities for $\frak a_1$, $\frak a_2$, $\frak a_3$ 
satisfying this bound.  

Naturally we may suppose that $N(\frak a_1) \le N(\frak a_2) \le N(\frak a_3)$.   It follows that $N(\frak a_1) \le 1/9+ \delta$, so that $N(\frak a_1)$ is either $1/12$ 
or $1/9$.   If $N(\frak a_1) =1/9$, then both $N(\frak a_2)$ and $N(\frak a_3)$ are also forced to be $1/9$, so that $\frak a_1$, $\frak a_2$, $\frak a_3$ are 
all either $[1,-3]$, or $[1,-2,-3, 6]$.   But in this case, it is not possible to have $h(\frak a) =2$.   We conclude that $N(\frak a_1) = 1/12$, so that $\frak a_1$ is 
either $[1]$ or $[1,-2]$.  

Next we must have $N(\frak a_2) \le (1/3+2\delta - 1/12)/2 = 1/8+\delta$, so that we must be in one of the following three cases: 
 $$ 
\text{Case I}:  \qquad \frak a_2= [1], [1,-2] \qquad N(\frak a_2)= 1/12,
$$ 
$$ 
\text{Case II}:  \qquad \frak a_2 = [1,-3], [1,-2,-3,6] \qquad N(\frak a_2) =1/9, 
$$  
$$
\text{Case III}:  \qquad \frak a_2  = [1,-4], [1,-2, 4] \qquad N(\frak a_2) =1/8.
$$

\subsection{Case I analysis}  Note that $N(\frak a_3) \le 1/3 + 2\delta -1/12 -1/12 = 1/6+2\delta$, and Lemma \ref{lem2.1} now gives the various possibilities 
for $\frak a_3$.    

If $\ell(\frak a_3) = 1$ (so that $\frak a_3=[1]$) or $\ell(\frak a_3)= 2$ (so that $\frak a_3=[a,b]$ with for coprime integers $a$ and $b$) then  the resulting 
lists $\frak a$ are all included in the family of multinomial coefficients 
\begin{equation} 
\label{3.1} 
[a+b+c, -a, -b, -c] = [a+b+c, -a, -b-c] + [b+c,-b,-c].  
\end{equation} 
This is a three parameter family, which is clearly reducible to two binomial coefficients.   

\smallskip

Now suppose that $\ell(\frak a_3)=3$.    Here $\frak a_3$ must be of the form $[a,-2a, b]$, or one of $[1, -3,9]$, $[1,-3,-4]$,  $[3, 4, -12]$.  Since $\frak a$ must 
have height $2$, we are forced to have one of $\frak a_1$ or $\frak a_2$ be $[1]$ and the other $[1,-2]$; say, $\frak a_1 = [1]$ and $\frak a_2=[1,-2]$.    
If $\frak a_3$ is of the form $[a,-2a, b]$, then a little calculation shows that $\frak a$ must belong to the three parameter, reducible family 
\begin{equation} 
\label{3.2} 
[2a, -a, 2b, -b, -c, -(a+b-c)] = [2a, -a, 2b, -b, -a-b]  +  [a+b, -c, - (a+b-c)].
\end{equation}   
In order for the left side of \eqref{3.2} to be a list of height $2$, we must have $c$ being a positive integer with $c < a+b$.  In the sequel, such conditions 
will be left implicit.   

The three other cases of length $3$, namely $\frak a_3= [1, -3, 9]$, $[1,-3,- 4]$, or $[3,4,-12]$ lead to the following three reducible, two parameter, families:  
\begin{equation} 
\label{3.3} 
[3a, -a, -9a, 2b, -b, 7a-b]  = a [3, 14, -1, -7, -9] + [7a, 2b, -14a, -b, 7a-b],  
\end{equation} 
 \begin{equation}
\label{3.4} 
[a,-3a,-4a, 2b, -b, 6a-b]   = a  [1, 12, -3, -4, -6] + [6a, 2b, -b, -12a, 6a-b], 
\end{equation} 
\begin{equation} 
\label{3.5} 
[12a, -3a, -4a, 2b,-b, -(b+5a)] = a[12, 5, -3,-4,-10] + [10a, -5a, 2b, -b, -b-5a]. 
\end{equation}

\medskip 

Now suppose $\ell(\frak a_3)=4$.  Then $\frak a_3 =[2a, 2b, -a,-b]$, or is given by one of the seven lists with length $4$, sum not equal to $0$, and norm $1/6$ given in Lemma \ref{lem2.1}.   In all these cases $\frak a_3$ has height $0$, and therefore we must have $\frak a_1 = \frak a_2 = [1]$.   
 The case $\frak a_3=[2a, 2b, -a, -b]$ leads to the family $[2a, 2b, -a, -b, -c, -d]$ with $c+d=a+b$, which is already included above, see \eqref{3.2}.     Thus we are 
 left with the following seven possibilities for $\frak a_3$:  
$$ 
[1,-3,6,-12]; \ \ [1, -3, -4, 12]; \ \ [1,-2, 4, -12];
$$
$$
  [1, -2, -3, 12]; \ \ [1,-4, -6, 12]; \ \ [2, -3, -4, 12]; \ \ [3, -4, -6, 12]. 
$$
Each of these seven lists may be completed to a five term list with sum $0$ which corresponds to a factorial ratio.   Therefore the families for $\frak a$ that we obtain from these lists are then all reducible, arising from one of these sporadic lists (with $D=1$) combined with a binomial coefficient.   Thus we obtained five new reducible, two parameter, families: 
\begin{equation} 
\label{3.6} 
[3a, 12a, -a, -6a, -b, -(8a-b)] =a[3,12,-1,-6,-8] + [8a, -b, b-8a], 
\end{equation} 
\begin{equation} 
\label{3.7} 
[2a,12a,-a,-4a, -b, -(9a-b)] = a[2,12, -1,-4,-9] + [9a, -b, b-9a],
\end{equation} 
\begin{equation} 
\label{3.8} 
[a,12a, -2a, -3a, -(8a-b)] =a [1,12, -2, -3, -8] + [8a,-b,b-8a],
\end{equation} 
\begin{equation} 
\label{3.9}
[2a,12a,-3a,-4a,-b,-(7a-b)] = a[2,12, -3,-4,-7] + [7a,-b,b-7a], 
\end{equation} 
\begin{equation} 
\label{3.10} 
[3a,12 a, -4a, -6a, -b, -(5a-b)] = a[3, 12, -4, -5, -6] + [5a,-b,b-5a]. 
\end{equation} 

\medskip

Now suppose $\ell(\frak a_3)= 5$, so that by Lemma \ref{lem2.1}, $\frak a_3$ must be one of the following four lists:
$$ 
[1,-2,-3,4,6]; \ \ [1,-2,-3,6,-12]; \ \ [2,3,-4,-6,12]; \ \ [1,-2, 4,6,-12]. 
$$ 
In all these cases we may suppose that $\frak a_1 = [1]$ and $\frak a_2 = [1,-2]$ because $\frak a$ must have height $2$.   
Then these four cases lead to the following reducible, two parameter, families: 
\begin{equation} 
\label{3.11} 
[2a, 3a, -a, - 4a, -6a, 2b, -b, 6a-b] = a [2, 3, 12, -1, -4, -6, -6 ] + [2b, -b, 6a, -12 a, 6a-b],
\end{equation} 
\begin{equation} 
\label{3.12}  
[a,6a,-2a,-3a,-12a, 2b,-b,10a-b] =a[1, 6, 20, -2, -3, -10, -12] + [2b, -b, 10a, -20a, 10a-b],
\end{equation} 
\begin{equation} 
\label{3.13} 
[4a, 6a, -2a, -3a, -12 a, 2b, -b, 7a-b] = a[4, 6, 14, -2, -3, -7,-12] + [7a , -14a, 2b, -b, 7a-b], 
\end{equation} 
 \begin{equation} 
\label{3.14} 
[2a, 12a, -a, -4a, -6a, 2b, -b, -3a-b] = a[2,3,12,-1,-4,-6,-6] + [6a,-3a,2b,-b,-3a-b].
\end{equation}

\medskip 

By Lemma \ref{lem2.1}, the last remaining cases are when $\ell(\frak a_3)=6$, and $\frak a_3$ is either  
$[1,-2,-3,6,9,-18]$, or  $[1,-2,-3,4,6,-12]$.   Since these lists have height $0$, we must have $\frak a_1= \frak a_2 =[1]$.  
Each of these possibilities for $\frak a_3$ can be completed to a $7$ term list with sum $0$, which forms a factorial ratio with $D=1$.  Thus, 
we get two more reducible, two parameter, families: 
\begin{equation} 
\label{3.15} 
[2a, 3a, 18a, -a, -6a, -9a, -b, b-7a] =a[2, 3, 18, -1, -6, -7, -9] + [7a, -b, b-7a], 
\end{equation} 
 \begin{equation} 
\label{3.16} 
[2a,3a, 12a, -a, -4a,-6a, -b, b-6a] =a[2, 3, 12, -1, -4, -6, -6] + [6a, -b, b-6a]. 
\end{equation} 

Thus Case I led to sixteen families of solutions, all of which are reducible.

\subsection{Case II analysis}  

Here $\frak a_2 = [1,-3]$ or $[1,-2,-3, 6]$ has norm $1/9$, so that $\frak a_3$ has 
norm in the range $1/9$ to $5/36$.  The possibilities for $\frak a_3$ are thus limited to the examples in Lemma \ref{lem2.1}, and indeed to just the cases 
$[a,b]$, $[a,-2a, b]$ and $[a,-2a, b,-2b]$.   The length $4$ case is ruled out as its height is $0$, and it would be impossible to have $\frak a$ of height $2$.  
The case $\frak a_3 =[a,b]$ can only arise with $a$ and $b$ of opposite sign (else the norm will exceed $1/6$), and again it is impossible to have $\frak a$ of 
height $2$.   We are left with $\frak a_3$ being of the form $[a,-2a,b]$, which gives the following five possibilities: 
 $$ 
[1,-2,4], \ \ [1,-2, -3], \ \ [2, 3, -6], \ \ [1, -3, 6], \ \ [1, -2, 6]. 
$$ 
In order for $\frak a$ to have height $2$, we must have $\frak a_1 =[1]$, and thus each of these five possibilities gives rise 
to two families for $\frak a$, corresponding to the two choices, $[1,-3]$ and $[1,-2,-3,6]$, for $\frak a_2$.   Going over the five possibilities for $\frak a_3$ in 
order, we find the following $10$ two parameter families:  
 \begin{equation}
 \label{3.17} 
[2a, b, 6b, -a, -4a, -2b, -3b, -(2b-3a)],
\end{equation} 
\begin{equation} 
\label{3.18} 
[2a, 3b, -a, -4a, -b, -(2b-3a)],  
\end{equation} 
\begin{equation} 
\label{3.19} 
[a, b, 6b, -2a, -3a, -2b, -3b, -(2b-4a)], 
\end{equation} 
\begin{equation}
\label{3.20} 
[a, 3b, -2a, -3a, -b, -(2b-4a)],
\end{equation} 
\begin{equation} 
\label{3.21}
[6a, 2b, 3b, -2a, -3a, -b, -6b, 2b-a],
\end{equation} 
\begin{equation} 
\label{3.22} 
[6a, b, -2a, -3a, -3b, 2b-a],
\end{equation} 
\begin{equation} 
\label{3.23}
[3a, b, 6b, -a, -6a, -2b, -3b, 4a-2b],
\end{equation} 
\begin{equation} 
\label{3.24}
[3a, 3b, -a, -6a, -b, 4a-2b], 
\end{equation} 
\begin{equation} 
\label{3.25}
[2a, b, 6b, -a, -6a, -2b, -3b, 5a-2b],
\end{equation} 
\begin{equation} 
\label{3.26}
[2a, 3b, -a, -6a, -b, 5a-2b]. 
\end{equation}

\subsection{Case III analysis}  

Here both $\frak a_2$ and $\frak a_3$ are either $[1,-4]$ or $[1,-2,4]$.  
Both cannot be $[1,-4]$, since then $\frak a$ cannot have height $2$.   So there are two cases: 
both are $[1, -2, 4]$ and $\frak a_1 = [1,-2]$; or $\frak a_1= [1]$, $\frak a_2= [1,-4]$, and $\frak a_3= [1,-2,4]$.  
These lead to two further two parameter families: 
\begin{equation} 
\label{3.27} 
[2a, 4b, -a, -4a, -b, 3a-3b],
\end{equation} 
\begin{equation} 
\label{3.28} 
[2a, 2b, 6(a+b), -a, -4a, -b, -4b, -3(a+b)].
\end{equation}

\medskip 

To sum up, we have shown that (apart from finitely many exceptions) lists of height $2$ and norm at most $1/3+\delta$ 
must belong to one of the $28$ families catalogued above.  The $16$ families of Section 3.1 are reducible, and every element in 
them with height $2$ corresponds to an integral factorial ratio.  To complete the proof of Theorem \ref{thm1}, it remains 
to show that the $12$ families described in Sections 3.1 and 3.2 are irreducible (see Corollary \ref{cor1}), and that lists of height $2$ in these families correspond to 
integral factorial ratios (see Section 7).

\section{Proof of Theorem \ref{thm2}} 

\noindent We begin by recalling that to any list $\frak a = [a_1,\ldots, a_n]$, we associate the function 
$\frak a(x) = \sum_{i=1}^{n} \psi(a_i x)$ (away from points $a_i x\in {\Bbb Z}$), which is odd and periodic with period $1$.  If the list $\frak a$ has 
sum $0$ and height $D$, then it is an integral factorial ratio precisely when $\frak a(x)$ takes values in the set $\{ -D/2, -D/2 +1, \ldots, D/2\}$.  In the sequel, we 
shall check this criterion for $\frak a(x)$ implicitly keeping $x$ away from points of discontinuity; right continuity will then ensure the result for all $x$.

For brevity, put $u= s(\frak a)$ and $v= s(\frak b)$.   The assumption that $v\frak a - u \frak b$ is an integral factorial ratio of height $D$ 
implies that for all real $x$ 
\begin{equation} 
\label{4.1} 
\frak{a}(vx) + \frak b(-ux) \in \{ -D/2, -D/2 +1, \ldots, D/2\}.
\end{equation} 
We shall show that for all $x$ and $y$ one has 
\begin{equation} 
\label{4.2} 
\frak a(x) + \frak b(y) + \psi(-ux-vy) \in \{ -(D+1)/2, -(D-1)/2, \ldots, (D+1)/2\}.
\end{equation} 
The theorem then follows upon applying the Landau criterion to compute the power of a prime 
dividing the numerator and denominator of the claimed integral factorial ratio.

Replacing $x$ by $vx$, and $y$ by $-u(x+t)$, we see that \eqref{4.2} is equivalent to the assertion that 
\begin{equation} 
\label{4.3} 
\frak a (vx) + \frak b(-u(x+t)) + \psi( uv t) \in \{ -(D+1)/2, -(D-1)/2, \ldots, (D+1)/2\}.
\end{equation} 
Using the monotonicity of $\frak b$, we shall reduce the above assertion to proving 
(for all real $x$ and integers $k$) 
\begin{equation} 
\label{4.4} 
\frak a(vx) + \frak b (- u(x+k/uv)) \in \{ - D/2, \ldots, D/2\}. 
\end{equation} 
Postponing the proof of this reduction, we now establish \eqref{4.4}.  Since $u$ and $v$ are coprime, we may 
write $k/uv = m/u + n/v$ for suitable integers $m$ and $n$.  Then, since $\frak a(vx)$ is periodic in $x$ with period $1/v$ 
and analogously for $\frak b(-ux)$,  
$$ 
\frak a(vx) + \frak b(-u(x+m/u+n/v)) = 
\frak a(vx) + \frak b(-u (x+n/v)) = \frak a(v(x+n/v)) + \frak b(-u(x+n/v)), 
$$ 
and so \eqref{4.4} follows from the assumption \eqref{4.1}.

We now prove that \eqref{4.4} implies \eqref{4.3}.  
The left side of \eqref{4.3} takes values in the set $(D+1)/2 + {\Bbb Z}$, and changes sign when $(x,t)$ is replaced 
by $(-x,-t)$.  Therefore it suffices to establish that the left side of \eqref{4.3} always takes values $\ge -(D+1)/2$, or 
that it always takes values $\le (D+1)/2$.    

Suppose that $uv >0$.  Here we shall show from \eqref{4.4} that the left side of \eqref{4.3} always takes 
values $\le (D+1)/2$.  In the case $uv <0$, the analogous argument shows that the left side of \eqref{4.3} always 
takes values $\ge -(D+1)/2$.   Suppose $k/uv \le t < (k+1)/uv$.  From the monotonicity of $\frak b$ (and note that 
the associated function in Definition 1.3 is increasing or decreasing depending on the sign of $v$) we see that 
$$ 
\frak b(-u(x+t)) + \psi(uvt) \le \frak b(-u(x+k/(uv))) + uv (t-k/uv) + \psi(uv t) = 1/2 + \frak b(-u(x+k/uv)). 
$$ 
Therefore, given \eqref{4.4} it follows that 
$$ 
\frak a(vx) + \frak b(-u(x+t)) + \psi(uv t) \le \frak a(vx) + \frak b(-u(x+k/uv)) + 1/2 \le (D+1)/2, 
$$
as needed.

This completes our proof of Theorem \ref{thm2}.  

\section{Examples arising from Theorem \ref{thm2}} 

\noindent In this section we give examples of factorial ratios of height $2$ obtained using 
Theorem \ref{thm2}.   The table gives a monotone list $\frak b$, a primitive list $\frak a$, 
and these lists satisfy the condition $(s(\frak a), s(\frak b)) =1$, and the table also displays 
the list $s(\frak b) \frak a - s(\frak a) \frak b$ which corresponds to an integral factorial ratio with height $1$.   
Thus each line of the table produces a two parameter family of integral factorial ratios with height $2$; for example, line 10 
shows that $[6a,-2a,-3a,b,-6b,-(a-5b)]$ is a factorial ratio of height $2$ provided $a > 5b$.  
We have not included in this table six further examples of Theorem \ref{thm2}; namely, the examples 
corresponding to the families \eqref{3.17}, \eqref{3.18}, \eqref{3.21}, \eqref{3.22}, \eqref{3.25}, and \eqref{3.26}. 

\vfill
\begin{center} 
\begin{tabular} {|c|| c| c| c|}
\hline
Number & List $\frak a$ & Monotone  $\frak b$ & Height $1$ factorial ratio \\
\hline
$1$ & $[2,-3,-4]$ & $[3, -1]$ & $[4,15,-5,-6,-8]$ \\ 
$2$ & $[10,-5,-6]$ & $[3,-1]$ & $[3,20,-1,-10,-12]$ \\ 
$3$ &$[6,-3,-4]$ &$[3, -1]$ & $[3,12,-1,-6,-8]$ \\  
$4$ &$[10,-2,-5]$& $[3, -1]$ & $[3,20,-4,-9,-10]$ \\ 
$5$ & $[10,-4,-5]$ & $[3, -1]$ & $[1,20,-3,-8,-10]$ \\ 
$6$ & $[6,-1,-4]$ & $[3, -1]$ & $[1,12, -2, -3,-8]$  \\ 
$7$ & $[4,-1,-2]$ & $[3, -1]$ & $[1,12,-3,-4,-6]$ \\  
$8$ & $[6,-3,-4]$ & $[4, -1]$ & $[4,18,-1,-9,-12]$ \\  
$9$ & $[6,-2,-3]$ & $[5, -1]$ & $[1,24,-5,-8,-12]$ \\ 
$10$ & $[6,-2,-3]$ & $[6, -1]$ & $[1,30,-6,-10,-15]$ \\ 
$11$ & $[3,10,-1,-5,-6]$ & $[3, -1]$ & $[1,6,20, -2,-3,-10,-12]$ \\ 
$12$ & $[2, 15,-1,-5,-6]$  & $[3, -1]$ & $[4,5,30,-2,-10,-12,-15]$ \\
$13$ & $[2,9,-1,-3,-4]$ & $[3, -1]$ & $[3, 4, 18, -2,-6,-8,-9]$ \\ 
$14$ & $[1,6,-2,-3,-3]$ & $[3, -1]$ & $[2,3,12,-1,-4,-6,-6]$ \\ 
$15$ & $[1,10,-3,-4,-5]$ & $[3, -1]$ & $[2,3,20,-1,-6,-8,-10]$ \\
$16$ & $[2, 15, -3, -4, -5]$ & $[3,-1]$ & $[4, 5, 30, -6,-8,-10,-15]$ \\ 
$17$ & $[3,10,-2,-5,-9]$ & $[3,-1]$ & $[6, 9, 20, -3, -4, -10,-18]$\\ 
$18$ & $[2, 12, -1,-4,-6]$ & $[3,-1]$ & $[3, 4, 24, -2, -8, -9, -12]$ \\ 
$19$ & $[2,3, 12,-1,-4,-6,-9]$ & $[3,-1]$ & $[4,6,9,24,-2,-3,-8,-12,-18]$\\ 
$20$ & $[2,-3,-4]$ & $[1, 6, -2, -3]$ & $[4, 5, 30, -6, -8, -10, -15]$\\ 
$21$ & $[10,-5,-6]$ & $[1,6,-2,-3]$ & $[1,6,20,-2,-3,-10,-12]$ \\ 
$22$ & $[6,-3,-4]$ & $[1,6,-2,-3]$ & $[1,12, -2,-3,-8]$ \\ 
$23$ & $[10,-2,-5]$ & $[1,6,-2,-3]$ & $[6,9,20, -3,-4,-10,-18]$ \\ 
$24$ & $[10,-4,-5]$ & $[1,6,-2,-3]$ & $[2, 3, 20, -1, -6,-8,-10]$ \\ 
$25$ & $[6,-1,-4]$ & $[1,6, -2,-3]$ & $[3,12,-1,-6,-8]$ \\ 
$26$ & $[6,-2,-3]$ & $[1,10,-2, -5]$ & $[2, 5, 24,-1,-8,-10,-12]$ \\
$27$ &$[3, 10, -1, -5,-6]$ & $[1,6,-2,-3]$ & $[3, 20, -1, -10, -12]$ \\ 
$28$ & $[2, 15, -1, -5, -6]$ & $[1,6, -2,-3]$ & $[4, 15, -2,-5,-12]$ \\ 
$29$ & $[2,9,-1,-3,-4]$ & $[1,6,-2,-3$ & $[4,9,-2,-3,-8]$ \\ 
$30$ & $[1,6,-2,-3,-3]$ & $[1,6,-2,-3]$ & $[1,12,-3,-4,-6]$ \\ 
$31$ & $[1,10,-3,-4,-5]$ & $[1,6,-2,-3]$ & $[1,20,-3,-8,-10]$ \\ 
$32$ & $[2,15,-3,-4,-5]$ & $[1,6,-2,-3]$ & $[4,15,-5,-6,-8]$ \\ 
$33$ & $[3,10,-2,-5,-9]$ & $[1,6,-2,-3]$ & $[3,20,-4,-9,-10]$ \\ 
$34$ & $[2,12,-1,-4,-6]$ & $[1,6,-2,-3]$ & $[4,6,9,24,-2,-3,-8,-12,-18]$\\  
$35$ & $[2,3,12,-1,-4,-6,-9]$ & $[1,6,-2,-3]$ & $[3,4,14,-2,-8,-9,-12]$\\
$36$ & $[4,-3]$ & $[1,4,-2]$ & $[2,12,-1,-4,-9]$ \\
$37$ & $[3,-2]$ & $[1,4,-2]$ & $[2,9,-1,-4,-6]$ \\
$38$ & $[1,4,-2,-2]$ & $[1,4,-2]$ & $[2,3,12,-1,-4,-6,-6]$ \\ 
$39$ & $[1,8,-3,-4]$ & $[1,4,-2]$ & $[3,4,24,-2,-8,-9,-12]$ \\
$40$ & $[2,3,8,-1,-4,-6]$ & $[1,4,-2]$ & $[4,6,9,24,-2,-3,-8,-12,-18]$ \\
$41$ & $[3,-2]$ & $[2,3,-1]$ & $[1,12,-2,-3,-8]$ \\
$42$ & $[3,-2]$ & $[1,6,-2]$ & $[2,15,-1,-6,-10]$ \\ 
$43$ & $[3,-2]$ & $[3,4,-2]$ & $[2,15,-3,-4,-10]$ \\ 
 \hline
\end{tabular} 
\end{center}

\section{The structure of reducible lists with $D=2$} 

Suppose $\frak a$ is a primitive list corresponding to a factorial ratio with $D=2$.   We wish to develop criteria to check whether the list 
$\frak a$ is irreducible.  



\begin{lemma} \label{lemr.1}  Suppose that $p \ge 11$ is a prime which divides some, but not all, elements of $\frak a$, and suppose that 
the multiples of $p$ in $\frak a$ do not sum to zero.    Then $\frak a$ cannot be decomposed as $\frak b + \frak c$ where both $\frak b$ and 
$\frak c$ are dilates of   sporadic integral factorial ratios of height $1$.   
\end{lemma} 
\begin{proof}  Suppose $\frak a$ can be decomposed as $\frak b +\frak c$.   The primitive sporadic factorial  ratios with $D=1$ have all elements divisible only by the primes $2$, $3$, $5$, $7$.  Therefore if either $\frak b$ or $\frak c$ contains a multiple of $p$ then all elements of that list must be multiples of $p$.  Since $\frak a$ is primitive, the elements of the other list must be coprime to $p$.    Therefore the multiples of 
$p$ in $\frak a$ must sum to zero, which we assumed not to be the case.   
\end{proof}

\begin{lemma} \label{lemr.2}   Suppose that $p \ge 11$ is a prime which divides some, but not all, elements of $\frak a$, and suppose that 
the multiples of $p$ in $\frak a$ do not sum to zero.   Suppose $\frak a$ decomposes as $\frak b + \frak c$ where $\frak c$ is a dilate of a sporadic 
 factorial ratio with $D=1$, while $\frak b$ lies in one of the infinite families with $D=1$ (so either 
$\frak b$ is of the form $[a+b,-a,-b]$ or of the form $[2a,-a,2b,-b, -(a+b)]$).    Then one of the following three cases holds: 

(i).  The number of multiples of $p$ in $\frak a$ is either exactly $1$, or is even and at least $4$. 

(ii).  There are exactly two multiples of $p$ in $\frak a$ and these are of the form $-ap$, $2ap$.  

(iii).  There are exactly three elements of $\frak a$ that are not multiples of $p$, and these include a pair $-b$, $2b$ with the third non-multiple being $\equiv -b \pmod p$.   
\end{lemma}  
\begin{proof} Since $\frak c$ is a dilate of a sporadic factorial ratio with $D=1$, either $\frak c$ consists entirely of multiples of $p$, 
or entirely of elements coprime to $p$.  In either case, since the sum of multiples of $p$ in $\frak a$ is non-zero, the list $\frak b$ must 
contain some multiples of $p$ and some elements coprime to $p$.   

If $\frak b$ is a binomial coefficient, then from the above remark, $\frak b$ must contain exactly one multiple of $p$.   If $\frak c$ has 
no multiples of $p$, then $\frak a$ will be left with exactly $1$ multiple of $p$.   If $\frak c$ consists entirely of multiples of $p$, then 
$\frak a$ either has $\ell(\frak c)-1$ or $\ell(\frak c) +1$ multiples of $p$, and this is an even number at least $4$.  Thus we are in case (i).  

If $\frak b$ is of the form $[2a,-a,2b,-b, -(a+b)]$ then (again by our previous remark) either $\frak b$ contains exactly $1$ multiple of $p$, or 
has $2$ multiples of $p$.   If $\frak b$ contains exactly $1$ multiple of $p$, then the argument of the preceding paragraph shows that we are in 
case (i).   Suppose now that $\frak b$ contains two multiples of $p$, which must be a pair of the form $-ap$, $2ap$ for some integer $a$.  If 
$\frak c$ has no multiples of $p$, then these are the only multiples of $p$ in $\frak a$, and we are in case (ii) of the lemma.  Finally, if $\frak c$ consists 
entirely of multiples of $p$, then the three elements of $\frak b$ that are not multiples of $p$ must be left uncanceled in $\frak a$, and these include a 
pair of elements $-b$, $2b$.  Thus we must be in case (iii) here.  
\end{proof}

\begin{lemma} \label{lemr.3}  Suppose that $p\ge 11$ is a prime, and that the number of multiples of $p$ in $\frak a$ is odd and 
at least $3$.  Suppose that the sum of the multiples of $p$ in $\frak a$ is not zero.   Suppose $\frak a$ decomposes as $\frak b + \frak c$ 
where both $\frak b$ and $\frak c$ belong to one of the infinite families with height $1$.    Then one of the following cases occurs: 



(i).  There are exactly three non-multiples of $p$ in $\frak a$, and when reduced $\pmod p$ these three 
elements are congruent to $x$, $x$, $-2x \pmod p$ for some non-zero $x\pmod p$.  

(ii).  There are exactly five non-multiples of $p$ in $\frak a$, and these element are of the form $4x$, $-x$, $2y$, $-y$, $-z$ for integers $x$, $y$, $z$.  
There are three multiples of $p$ in $\frak a$, and these are either $2z-4x$, $-(z-2x)$, $-(x+y)$, or $2(z-x)$, $-(z-x)$, $-(2x+y)$.    

(iii).   There are exactly five non-multiples of $p$ in $\frak a$, and these are elements of the form $x$, $2y$, $-y$, $2z$, $-z$.  There are three 
multiples of $p$ in $\frak a$, and these are either $-(x/2+y)$, $-(x+2z)$, $(x/2+z)$ (and this only occurs for $x$ even), or $(x-y)$, $-2(2x+z)$, $(2x+z)$.  

(iv). There are no degeneracies in concatenating $\frak b$ and $\frak c$, and 
either $\frak a = [2a, -a, 2b, -b, 2c, -c, 2d, -d, -(a+b), -(c+d)]$, or $\frak a = [2a, -a, 2b, -b, -(a+b), (c+d), -c, -d]$.  
\end{lemma}
\begin{proof}  If either $\frak b$ or $\frak c$ has no multiples of $p$ and the other list is entirely composed of multiples of $p$, then the 
sum of multiples of $p$ in $\frak a$ would be zero.  Thus, this case is forbidden.  Further, if $\frak b$ has $u$ multiples of $p$ and $\frak c$ has 
$v$ multiples of $p$, then the number of multiples of $p$ in $\frak a$ is at most $u+v$ and has the same parity as $u+v$.  Thus we may restrict attention 
to the cases when $u+v$ is odd and at least three.  We will make repeated use of these observations below.   Indeed, these observations immediately 
rule out the possibility that both $\frak b$ and $\frak c$ are binomial coefficients (since any binomial coefficient would have $0$, $1$ or $3$ multiples of $p$).   
We are left with two cases: by symmetry we assume that $\frak c$ is of the form $[2a,-a,2b,-b,-(a+b)]$, and $\frak b$ is either also of this form, or $\frak b$ is a 
binomial coefficient.

\smallskip

\noindent {\bf The case $\frak b$ is a binomial coefficient and $\frak c$ is of the form $[2a,-a, 2b, -b, -(a+b)]$.}  Then $\frak b$ has $0$, $1$ or $3$ multiples 
of $p$, and $\frak c$ has $0$, $1$, $2$, or $5$ multiples of $p$.   Using our observations above, we are reduced to two possibilities:  
  $\frak b$ has $1$ or $3$ multiples of $p$, and $\frak c$ has $2$ multiples of $p$.  


 Suppose $\frak b$ has $3$ multiples of $p$ and $\frak c$ has $2$ multiples of $p$.  
  Then there are three non-multiples of $p$ in $\frak c$, which are left uncanceled in $\frak a$.   These elements in $\frak c$ must sum to zero $\pmod p$, and 
include a pair $-a$, $2a$, so that we are in case (i). 

Now suppose $\frak b$ has exactly $1$ multiple of $p$ and $\frak c$ has $2$ multiples of $p$.   Since $\frak a$ has at least $3$ multiples of $p$, 
there is no cancelation among the multiples of $p$ in $\frak b$ and $\frak c$.     If there is no cancelation among the non-multiples of $p$ in $\frak b$ and $\frak c$ as 
well, then we must be in case (iv).    Suppose then that there is some cancelation among the non-multiples of $p$ in $\frak b$ and $\frak c$.      Now the non-multiples of $p$ in $\frak b$ look like $-y$, $y \pmod p$ for some 
$y\not\equiv 0\pmod p$, and the non-multiples of $p$ in $\frak c$ look like $2x$, $-x$, $-x \pmod p$ for some $x\neq 0 \pmod p$.  It follows that there must be exactly one non-multiple in $\frak b$ that cancels with a non-multiple in $\frak c$.  After canceling them, we must be left with three non-multiples that look like $2x$, $-x$, $-x \pmod p$.   That is, we are in case (i).

\smallskip 

\noindent{\bf Both $\frak b$ and $\frak c$ are of the form $[2a,-a,2b,-b,-(a+b)]$.}   Suppose, by symmetry, that $\frak c$ has at least as many multiplies 
of $p$ as $\frak b$.  Culling the possibilities for the number of multiples of $p$ using our earlier observations, we are left with two choices:  
$\frak b$ has $1$ multiple of $p$ and $\frak c$ has $2$ multiples of $p$, 
or $\frak b$ has $2$ multiples of $p$ and $\frak c$ has $5$ multiples of $p$.  In the second case, there are $3$ non-multiples of $p$ in $\frak a$, and we 
are in case (i).

We are left with the case that $\frak b$ has $1$ multiple of $p$ and $\frak c$ has $2$ multiples of $p$, and we may assume that there 
is no cancelation among these multiples of $p$.   If there is no cancelation among the non-multiples of $p$ in $\frak b$ and $\frak c$ then we are in case (iv).  
So there must be some cancelation among the non-multiples of $p$ in $\frak b$ (which we write as $2a$, $-a$, 
$2b$, $-b$ with $a+b\equiv 0 \pmod p$) and $\frak c$ (which we write as $2c$, $-c$, $-d$ with $c\equiv d \pmod p$).  If $c=-a$ or $-b$ then we are in case (i).  
If $2c=a$ or $b$, or $c=2a$ or $2b$ then a quick check shows that we are in case (ii). 
If $d=-a$ or $d=-b$, or $d=-2a$ or $d=-2b$ then we are in case (iii).

Having exhausted all possibilities, the proof of the lemma is complete.
\end{proof}

\begin{corollary}  \label{cor1}   The twelve families given in \eqref{3.17} to \eqref{3.28}, together with the $43$ families listed in the table in Section 5 
are all irreducible.   Askey's family, 
$$ 
[3(m+n), 3n, 2m, 2n, -(2m+3n), -(m+2n), -(m+n), -m, -n, -n], 
$$ 
is also irreducible.  
\end{corollary}  
\begin{proof}   Apart from Askey's family and the example \eqref{3.28}, the remaining $54$ families look like $a \frak a + b\frak b + (-a s(\frak a) - bs(\frak b)) [1]$ 
for suitable primitive lists $\frak a$, $\frak b$, where $a$ and $b$ are coprime, and chosen so that the resulting list has height $2$.  In all these examples, $s(\frak a)$ and $s(\frak b)$ are both non-zero, and either $\ell(\frak a)$ or $\ell(\frak b)$ is an 
odd number at least $3$.  If $\frak a$ has an odd number of elements, then choose $a$ with large size and divisible by $p$ for some prime $p\ge 11$, and similarly if $\frak b$ has 
an odd number of elements, choose $b$ with large size and divisible by $p$.   Lemma \ref{lemr.1} now guarantees that such a list cannot be decomposed into two dilates 
of sporadic factorial ratios of height $1$.   By straightforward (if lengthy) inspection, we can eliminate the various cases that reducible lists must belong 
to (given in Lemmas \ref{lemr.2} and \ref{lemr.3}), and conclude that all these lists are irreducible.  

The argument for list \eqref{3.28} is similar, choosing $a$ to be  a large multiple of a prime $p\ge 11$, and checking the conclusions of Lemmas \ref{lemr.1}, \ref{lemr.2}, and \ref{lemr.3}.   

In Askey's family, we choose $m$ and $n$ to be large coprime positive numbers with $m+n$ being an odd multiple of a prime $p\ge 11$.   Lemmas \ref{lemr.1} and \ref{lemr.2} still apply and shows that such a list is not reducible into two sporadic factorial ratios, or into a sporadic factorial ratio and one from an infinite family.   Since the lists in Askey's family have length $10$, the only remaining possibility is that the list looks like $[2a,-a, 2b, -b, 2c, -c, 2d, -d, -(a+b), -(c+d)]$.   But such lists (with height $2$) have 
the property that the largest element in them is even, whereas our example from Askey's family has largest element $3(m+n)$ which is odd.
\end{proof}

 \section{Completing the proof of Theorem \ref{thm1}} 
 
 \noindent We have already shown that, apart from finitely many exceptions, all primitive lists with height $2$ and norm at most $1/3+ \delta$ 
 must lie in one of the $28$ families catalogued in Section 3.  The $16$ families given in Section 3.1 are reducible, and thus the lists of height $2$ in 
 these families correspond automatically to integral factorial ratios.   The $12$ families given in Sections 3.2 and 3.3 are all known by Corollary \ref{cor1} to 
 be irreducible.    In Section 5 we noted that the height $2$ lists from the 
 families \eqref{3.17}, \eqref{3.18}, \eqref{3.21}, \eqref{3.22}, \eqref{3.25}, and \eqref{3.26} are all integral factorial ratios 
 thanks to Theorem \ref{thm2}.   Thus all that remains is to show that height $2$ lists in the six families \eqref{3.19}, \eqref{3.20}, \eqref{3.23}, \eqref{3.24}, \eqref{3.27}, and
 \eqref{3.28} are also integral factorial ratios.

 Given a particular family it is straightforward to check whether the elements in it correspond to integral factorial ratios.  We illustrate with one of
 these six remaining families, the others following similarly.   To show that lists of height $2$ from \eqref{3.28} give rise to integral factorial ratios, it is 
 enough to show that 
 $$ 
 \lfloor 2x \rfloor - \lfloor x \rfloor -\lfloor 4x \rfloor + \lfloor 2y \rfloor - \lfloor y \rfloor -\lfloor 4y \rfloor + \lfloor 6(x+y) \rfloor - \lfloor 3(x+y) \rfloor \ge 0 
 $$ 
 for all $x$ and $y$.  Since the left side is periodic in $x$ and $y$ with period $1$, it is enough to verify the inequality for $x$ and $y$ in $[0,1)$.  
 For fixed $x$, the quantity $\lfloor 6(x+y)\rfloor -\lfloor 3(x+y)\rfloor$ is increasing in $y$, while the quantity $\lfloor 2y \rfloor -\lfloor y\rfloor -\lfloor 4y\rfloor$ is 
 constant on the intervals $[0,1/4)$, $[1/4,3/4)$, and $[3/4,1)$.  So it is enough to verify the inequality at $y=0$, $1/4$ and $3/4$.  Arguing similarly, it is enough to check 
 the inequality when $x=0$, $1/4$, $3/4$.  After a small calculation to check these nine cases, the inequality follows.

 \end{document}